\documentclass[11pt,twoside]{article}

\usepackage[T1]{fontenc}
\usepackage{lmodern}
\usepackage{amsmath,amssymb,amsthm,mathtools}
\usepackage{microtype}
\usepackage{geometry}
\usepackage{fancyhdr}
\usepackage{titlesec}
\usepackage{xcolor}
\usepackage[hypertexnames=false]{hyperref}
\usepackage{cleveref}

\definecolor{citationgreen}{RGB}{0,120,45}
\hypersetup{
  colorlinks=true,
  linkcolor=black,
  citecolor=citationgreen,
  urlcolor=citationgreen,
  pdfauthor={Guorong Gao and Mingze Li},
  pdftitle={On Non-Uniform Baranyai Factorizations},
  pdfsubject={An asymptotic classification of one-factorizations of fixed unions of Boolean-lattice levels},
  pdfkeywords={Baranyai theorem, nonuniform hypergraph, one-factorization, unions of levels, gcd peeling}
}

\titleformat{\section}[block]
  {\centering\normalfont\large\scshape}
  {\S\thesection.}{0.55em}{}
\titlespacing*{\section}{0pt}{3.0ex plus 1ex minus .2ex}{1.8ex plus .2ex}

\titleformat{\subsection}[block]
  {\normalfont\bfseries}
  {\thesubsection.}{0.5em}{}
\titlespacing*{\subsection}{0pt}{2.4ex plus .8ex minus .2ex}{1.2ex plus .2ex}

\newtheoremstyle{paperplain}
  {6pt}{6pt}{\itshape}{} {\bfseries}{.}{0.5em}{}
\theoremstyle{paperplain}
\newtheorem{theorem}{Theorem}[section]
\newtheorem{proposition}[theorem]{Proposition}
\newtheorem{lemma}[theorem]{Lemma}
\newtheorem{corollary}[theorem]{Corollary}
\newtheorem{conjecture}[theorem]{Conjecture}

\theoremstyle{definition}
\newtheorem{definition}[theorem]{Definition}

\numberwithin{equation}{section}

\newcommand{\Z}{\mathbb{Z}}

\newcommand{\T}{\mathcal{T}}

\newcommand{\Layers}[2]{\binom{[#1]}{#2}}

\let\leq\leqslant

\let\le\leqslant
\let\ge\geqslant

\fancypagestyle{plain}{
  \fancyhf{}
  \fancyfoot[C]{\small\thepage}

}

\makeatletter
\renewcommand{\maketitle}{%
  \begin{center}
    \vspace*{0.25in}
    {\Large\bfseries\@title\par}
    \vspace{1.6em}
    {\normalsize\@author\par}
  \end{center}
  \vspace{1.5em}
  \thispagestyle{plain}
}
\makeatother

\newenvironment{paperabstract}
  {\begin{list}{}{\leftmargin=0.42in\rightmargin=0.42in}
   \item[]\small\noindent\textsc{Abstract.}\ }
  {\end{list}\vspace{1.0em}}

\title{On Non-Uniform Baranyai Factorizations}
\author{Guorong Gao, Mingze Li}
\date{}

\begin{document}

\maketitle

\begin{paperabstract}
For a fixed nonempty finite set $L$ of positive integers and $k=\max L\ge2$, let $\Layers{n}{L}=\bigcup_{\ell\in L}\binom{[n]}{\ell}$ denote the non-uniform hypergraph whose edge set consists of all $\ell$-subsets of $[n]$ for $\ell\in L$. 
A $1$-factor of a hypergraph $H$ is a collection of pairwise disjoint edges whose union is $V(H)$. We say $H$ is $1$-factorable if $E(H)$ can be partitioned into $1$-factors. 
In this paper, we study the 1-factorability of $\Layers{n}{L}$. 
Earlier work of He, Huang, and Ma shows that, for all sufficiently large $n$, $1$-factorability can occur only in the residue classes $0$ and $-1$ modulo $k$: it always occurs in class $0$, whereas it occurs in class $-1$ only if $k-1\in L$.  We determine exactly when $1$-factorability occurs in this exceptional class. More precisely, we define a finite gcd-peeling algorithm on $L\setminus\{k-1,k\}$ and prove that $\Layers{qk-1}{L}$ is $1$-factorable if and only if the algorithm accepts. As a consequence, our result provides a criterion for the $1$-factorability of $\Layers{n}{L}$ and confirms the conjecture posed by He, Huang, and Ma.
\end{paperabstract}

\section{Introduction}

A hypergraph $H=(V,E)$ consists of a finite vertex set $V$ and a family $E$ of
nonempty subsets of $V$, called edges.  A \emph{$1$-factor} of $H$ is a
collection of pairwise disjoint edges whose union is $V$, and a
\emph{$1$-factorization} is a partition of $E$ into $1$-factors.  Equivalently,
the edges are colored so that every color class is a perfect matching.

For positive integers $k\le n$, the complete $k$-uniform hypergraph on $[n]$
has edge set $\binom{[n]}{k}$.  The divisibility condition $k\mid n$ is necessary for a $1$-factorization, and Baranyai's theorem says that it is also
sufficient \cite{Baranyai}.  In design-theoretic language, the $k$-subsets can
therefore be partitioned into parallel classes; see also
\cite{BrouwerSchrijver,Cameron} for the broader setting of parallelisms and
uniform hypergraph decompositions.

The natural nonuniform version fixes a nonempty finite set $L$ of positive
integers, puts $k:=\max L$, and considers the union of levels
\[
  \Layers{n}{L}:=\bigcup_{\ell\in L}\binom{[n]}{\ell}.
\]
This problem has a substantial history.  Berge and Johnson~\cite{BergeJohnson} used linear and integer programming to study hereditary closures of complete uniform hypergraphs; Johnson~\cite{Johnson} treated further closure cases. Brouwer and Tijdeman~\cite{BrouwerTijdeman} completely classified the initial segment $L=[k]$ and obtained exact results for several further families of unions of levels. He, Huang, and Ma~\cite{HeHuangMa} later gave a network-flow formulation, revisited the initial-segment case, and established asymptotic restrictions and a classification conjecture for arbitrary fixed $L$. Jiang, Sun, and Zhang~\cite{JiangSunZhang} studied multiplicity versions of the nonuniform problem.

\vfill
\noindent\rule{2.8in}{0.4pt}\par
\smallskip
{\footnotesize
\textit{2020 Mathematics Subject Classification.}
Primary 05C65; Secondary 05C70.\quad
\textit{Key words and phrases.}
Baranyai's theorem, nonuniform hypergraph, $1$-factorization, unions of levels,
gcd peeling.\par}
\newpage

The realization principle turns the factorization problem into an integer
system indexed by partition types.  This reduction was already recorded by
Brouwer and Tijdeman \cite[p.~242]{BrouwerTijdeman}; we use the network-flow
formulation of He, Huang, and Ma \cite[Theorem~2.2]{HeHuangMa}.
Bahmanian's detachment theorem gives a general realization result for
nonuniform complete multihypergraphs with prescribed regular degrees and
edge counts of each size \cite[Theorem~7.2]{BahmanianI}.  His recent sequel
also controls connectivity and multiplicities within each factor
\cite[Theorem~3.1]{BahmanianII}.  These results provide a general framework
for realizing feasible integer data.  The question addressed here is to
determine precisely when such data exist for $\Layers{qk-1}{L}$, by an
explicit arithmetic criterion depending only on $L$.

Recent work also studies extensions of prescribed factorizations.
Bahmanian \cite{BahmanianII} and Wang, Feng, and Wang \cite{WangFengWang}
give embedding results for regular colorings of complete uniform
multihypergraphs.  Bahmanian's work on symmetric layer-rainbow Latin cubes
\cite{BahmanianCubes} treats a specific nonuniform embedding problem with
edge sizes $1,2,3$ and respective multiplicities $1,3,2$.  These embedding
questions complement the existence classification for arbitrary fixed
sets $L$ considered here.

For arbitrary fixed $L$ with $k\ge2$, the asymptotic results of
He, Huang, and Ma imply that only the residue
classes $0$ and $-1$ modulo $k$ can occur, that the class $0$ always occurs,
and that the class $-1$ cannot occur unless $k-1\in L$.  They consequently
posed the following classification problem \cite[Conjecture~4.4]{HeHuangMa}.

\begin{conjecture}[He--Huang--Ma \cite{HeHuangMa}]
\label{conj:hhm}
Let $L$ be a fixed nonempty finite set of positive integers and let
$k=\max L\ge2$.  Exactly one of the following holds for all sufficiently
large $n$:
\begin{enumerate}
  \item[(i)] $\Layers{n}{L}$ is $1$-factorable if and only if
        $n\equiv0\pmod{k}$;
  \item[(ii)] $\Layers{n}{L}$ is $1$-factorable if and only if
        $n\equiv0$ or $-1\pmod{k}$.
\end{enumerate}
\end{conjecture}

The restriction $k\ge2$ only removes a linguistic degeneracy: when $k=1$ the
two displayed residue sets coincide.  

We first isolate the conclusions of He, Huang, and Ma that reduce
Conjecture~\ref{conj:hhm} to the single case treated here.

\begin{proposition}[He--Huang--Ma {\cite[Theorems~4.1--4.3]{HeHuangMa}}]
\label{prop:old-reduction}
For every fixed nonempty finite set $L$ with $k=\max L\ge2$, all sufficiently large
$n$ satisfy the following.
\begin{enumerate}
  \item If $n\equiv0\pmod{k}$, then $\Layers{n}{L}$ is $1$-factorable.
  \item If $\Layers{n}{L}$ is $1$-factorable, then
        $n\equiv0$ or $-1\pmod{k}$.
  \item If $n\equiv-1\pmod{k}$ and $k-1\notin L$, then
        $\Layers{n}{L}$ is not $1$-factorable.
\end{enumerate}
\end{proposition}
The unresolved case of
Conjecture~\ref{conj:hhm} is therefore $n=qk-1$ and $\{k-1,k\}\subseteq L$. Our main result decides this case by the following \emph{gcd-peeling algorithm}.

\begin{definition}[The gcd-peeling algorithm]
\label{def:peeling}
Given $M\subseteq[k-2]$, initialize
\[
  R:=M,\qquad C:=\varnothing,\qquad B:=\varnothing,\qquad m:=k.
\]
Repeat the following steps.
\begin{enumerate}
  \item If $R=\varnothing$, return \textsc{accept}.
  \item Put $h:=\max R$ and $d:=\gcd(m,h)$.
  \item If $d=1$, return \textsc{accept}.
  \item If $d>1$ and $h-1\notin R$, return \textsc{reject}.
  \item Otherwise record the bridge pair $(h,h-1)$ and replace
        \begin{equation}
          R\leftarrow R\setminus\{h,h-1\},\quad
          C\leftarrow C\cup\{h\},\quad
          B\leftarrow B\cup\{h-1\},\quad
          m\leftarrow d.
          \label{eq:peeling-update}
        \end{equation}
\end{enumerate}
\end{definition}

Each nonterminal pass deletes two elements of $R$, so the algorithm is finite.
Applying this algorithm, we obtain a criterion for the unresolved case, which constitutes the main result of this paper.

\begin{theorem}
\label{thm:main}
Let $L$ be a fixed nonempty finite set of positive integers and let
$k=\max L\ge2$.  Suppose that $k-1\in L$, and put
$M=L\setminus\{k-1,k\}$.  Then, for every integer $q\ge2k^2$,
\[
  \Layers{qk-1}{L}\text{ is $1$-factorable}
  \quad\Longleftrightarrow\quad
  M\text{ is accepted by the gcd-peeling algorithm}.
\]
\end{theorem}

Theorem~\ref{thm:main} settles exactly the only residue class left open by
Proposition~\ref{prop:old-reduction}.  Indeed, that proposition guarantees the
class $0\pmod{k}$ and excludes every class except possibly $-1\pmod{k}$; it
also excludes the latter class when $k-1\notin L$.  When $k-1\in L$,
Theorem~\ref{thm:main} decides this final class according to the peeling
output.  Combining these statements gives the following complete trichotomy.

\begin{corollary}
\label{cor:hhm-classification}
The nondegenerate $k\ge2$ form of Conjecture~\ref{conj:hhm} holds for every
fixed set $L$.  More precisely:
\begin{enumerate}
  \item[(i)] if $k-1\notin L$, alternative~{\rm(i)} holds;
  \item[(ii)] if $k-1\in L$ and the gcd-peeling algorithm accepts
        $L\setminus\{k-1,k\}$, alternative~{\rm(ii)} holds;
  \item[(iii)] if $k-1\in L$ and the algorithm rejects,
        alternative~{\rm(i)} holds.
\end{enumerate}
\end{corollary}

The criterion includes the previously known benchmark cases.  It accepts
$M=\varnothing$, recovering $L=\{k-1,k\}$, and it accepts $M=[k-2]$,
corresponding to $L=[k]$, whose exact classification was obtained by Brouwer
and Tijdeman \cite[Theorem~1]{BrouwerTijdeman} and revisited in
\cite{HeHuangMa}.  More generally, the case in which peeling exhausts $M$
follows from \cite[Theorem~8a and Proposition~2]{BrouwerTijdeman}.
For $k=4$ and
$L=\{2,3,4\}$, it rejects immediately, matching the classical obstruction.
When $k-2\in L$, the first step accepts at once if $k$ is odd.  If $k$ is
even, it rejects when $k-3\notin L$; otherwise it peels $(k-2,k-3)$ and
reduces the modulus to $2$.  In this last case the
subsequent steps pair consecutive odd and even levels, so the criterion
specializes to the buddy-parity condition of Brouwer and Tijdeman
\cite[Theorem~11]{BrouwerTijdeman}.

The proof has two complementary parts.  If peeling accepts, we explicitly
construct an integral solution of the type equations.  Bridge types absorb
the paired levels, one terminal congruence type handles the unpaired tail,
and a zero-sum decomposition closes the remaining coordinates.  If peeling
rejects, the final nontrivial gcd forces every type to use one of a prescribed
set of levels; a binomial-coefficient comparison then rules out even a
nonnegative real solution of the type equations. \Cref{sec:2} records the type reduction and the
elementary tools. \Cref{sec:3} proves Theorem~\ref{thm:main} and
Corollary~\ref{cor:hhm-classification}.

\section{Preliminaries}\label{sec:2}
\subsection{Types}

For $n\ge0$, write $[n]=\{1,\dots,n\}$, with $[0]=\varnothing$.

\begin{definition}
\label{def:types}
For $n\ge k$, an \emph{$(n,L)$-type} is a vector $\lambda=(\lambda_\ell)_{\ell\in L}\in\Z_{\ge0}^{L}$ such that $\sum_{\ell\in L}\ell\lambda_\ell=n$. Let $\T(n,L)$ be the finite set of all $(n,L)$-types. 
\end{definition}

The following theorem is the exact realization principle that we use as a
black box.  It is proved in \cite[Theorem~2.2]{HeHuangMa} by a network-flow
argument in the spirit of Baranyai's equalization method.

\begin{theorem}[Type realization \cite{HeHuangMa}]
\label{thm:type-realization}
The hypergraph $\Layers{n}{L}$ is $1$-factorable if and only if there are
nonnegative integers $(x_\lambda)_{\lambda\in\T(n,L)}$ satisfying
\begin{equation}
  \sum_{\lambda\in\T(n,L)}x_\lambda\lambda_\ell
  =\binom{n}{\ell}
  \qquad(\ell\in L).
  \label{eq:type-system}
\end{equation}
\end{theorem}

Thus it is enough to construct integral multiplicities of types; the theorem
then realizes those multiplicities by an actual $1$-factorization.

\subsection{Earlier asymptotic reductions and counting identities}

By Proposition~\ref{prop:old-reduction}, there is only one unresolved case. Henceforth fix
\[
  n=qk-1,\qquad \{k-1,k\}\subseteq L,\qquad
  M:=L\setminus\{k-1,k\}\subseteq[k-2].
\]
For $1\le i\le k$, abbreviate
\begin{equation*}
  H_i:=\binom{n}{i},\qquad
  D_i:=\binom{n-1}{i-1},\qquad
  S:=\sum_{i\in M}D_i,\qquad
  U:=H_{k-1}.
\end{equation*}
Double counting and Pascal's identity give
\begin{equation}
  iH_i=nD_i,\qquad H_i=D_i+D_{i+1}\quad(1\le i\le k-1),
  \label{eq:basic-identities}
\end{equation}
and therefore $D_{k-1}+D_k=U$.

Even if the multiplicities in \eqref{eq:type-system} are relaxed to be
nonnegative reals, their total mass is forced.  Indeed,
\begin{align*}
  n\sum_{\lambda}x_\lambda
  &=\sum_{\lambda}x_\lambda
    \sum_{\ell\in L}\ell\lambda_\ell
    =\sum_{\ell\in L}\ell H_\ell
    =n\sum_{\ell\in L}D_\ell,
\end{align*}
so we have 
\begin{equation}
  \sum_{\lambda}x_\lambda=S+U.
  \label{eq:total-mass}
\end{equation}
The types remain integral in this observation; only their multiplicities have
been relaxed.

\subsection{Congruences and zero-sum blocks}

For a finite set $A$ of integers, let $\gcd(k,A)$ denote the gcd of $k$ and
all elements of $A$, with $\gcd(k,\varnothing):=k$.

\begin{definition}
Let $(\mathcal{R},+,\cdot)$ be a ring. A sequence $\{a_{i}\}_{1\leq i\leq t}$ in $\mathcal{R}$ is called \emph{zero-sum} if $\sum_{i=1}^{t}a_{i}=0$.  
\end{definition}

\begin{lemma}
\label{lem:zero-sum}
Let $A$ be a finite set of positive integers.
\begin{enumerate}
  \item Put $g=\gcd(k,A)$.  The congruence $\sum_{a\in A}az_a\equiv r\pmod{k}$ has an integral solution if and only if $g\mid r$.  Whenever a
        solution exists, one exists with every $z_a\in\{0,\dots,k-1\}$.
  \item Every finite zero-sum sequence in $\Z/k\Z$ can be partitioned into
        nonempty zero-sum subsequences, each of length at most $k$.
\end{enumerate}
\end{lemma}

\begin{proof}
For the first statement, the case $A=\varnothing$ is immediate.  Otherwise
necessity is clear.  If $g\mid r$, Bezout's identity expresses $g$ as an
integral linear combination of $k$ and the elements of $A$.  Multiplying by
$r/g$ gives a solution, and reducing each coefficient modulo $k$ places it in
$\{0,\dots,k-1\}$.

For the second statement, order a nonempty zero-sum sequence as
$a_1,\dots,a_t$.  If $t\le k$, take the whole sequence.  If $t>k$, consider
the $k+1$ partial sums
\[
  0,\ a_1,\ a_1+a_2,\ \dots,\ a_1+\cdots+a_k
\]
in $\Z/k\Z$.  Either a nonempty initial sum is zero or two partial sums agree.
In either case there is a nonempty zero-sum subsequence of length at most $k$.
Remove it and apply induction to the remaining zero-sum sequence.
\end{proof}

We also repeatedly use the elementary monotonicity
\begin{equation*}
  \binom{N}{a}\le\binom{N}{b}
  \qquad\left(0\le a\le b\le\frac{N-1}{2}\right),
\end{equation*}
which follows from
$\binom{N}{t+1}/\binom{N}{t}=(N-t)/(t+1)\ge 1$ for $t\leq (N-1)/2$. In particular, when
$n\ge2k$, the sequences $D_i$ and $H_i$ are nondecreasing on every index
range used below.

\subsection{Peeling invariants}

\begin{lemma}
\label{lem:state-invariants}
At every visit to Step~1 of the gcd-peeling algorithm, we have $M=R\mathbin{\dot\cup}C\mathbin{\dot\cup}B$, $B=\{c-1:c\in C\}$, the map $c\mapsto c-1$ is a bijection from $C$ to $B$, and $m=\gcd(k,C)$. 
Moreover, every element of $C\cup B$ is larger than every element of $R$.
Consequently, if $R\ne\varnothing$ and $h=\max R$, then
\begin{equation}
  \gcd(m,h)=\gcd(k,C\cup\{h\}).
  \label{eq:state-terminal-gcd}
\end{equation}
The algorithm terminates.
\end{lemma}

\begin{proof}
All assertions hold initially.  Suppose a pair $(h,h-1)$ is peeled.  The two
removed elements are inserted into the disjoint sets $C$ and $B$, preserving
the partition and the stated bijection.  The update of $m$ gives $\gcd(m,h)=\gcd(k,C\cup\{h\})$, 
which is the new gcd identity.  Since $h=\max R$, every element left in $R$ is
at most $h-2$, whereas all previously peeled elements were already larger
than $h$.  Thus the order invariant is preserved.  Finally, each nonterminal
iteration removes two elements of $R$.
\end{proof}

\section{Proof of the classification}\label{sec:3}

In this section, we will prove Theorem~\ref{thm:main}. To this end, we divide the proof into two parts.

\subsection{Acceptance implies factorability}

\begin{theorem}\label{thm:proof1}
Under the same hypotheses as in Theorem~\ref{thm:main}, if $M$ is accepted by the gcd-peeling algorithm, then $\binom{[qk-1]}{L}$ is $1$-factorable. 
\end{theorem}

\begin{proof}
Assume that the algorithm accepts.  At its terminal state retain the notation
$R,C,B$ from Lemma~\ref{lem:state-invariants}.  A type with
$\lambda_{k-1}=0$ will be called a \emph{defect type}.

For a vector $P=(P_i)_{i\in M}\in\Z_{\ge0}^{M}$, put $w(P):=\sum_{i\in M}iP_i$.
Whenever $w(P)\equiv-1\pmod{k}$ and $w(P)\le n$, 
the extension
\begin{equation*}
  \widehat P_i=P_i\ (i\in M),\qquad
  \widehat P_{k-1}=0,
  \qquad
  \widehat P_k=\frac{n-w(P)}{k}
\end{equation*}
is an $(n,L)$-type.
We primarily consider two defect types: \emph{bridge patterns} and \emph{terminal patterns}. We present their constructions and determine their multiplicities. 

\medskip
\noindent\textit{Bridge patterns.}
For every $c\in C$, define $g_c:=\gcd(k,c)$ and $\alpha_c:=\frac{k}{g_c}-1$. 
When $c$ was peeled, $\gcd(m,c)>1$.  Since $m\mid k$, this implies $g_c>1$.
Also $c<k$, so $g_c<k$ and $\alpha_c$ is a positive integer.  Use
$H_{c-1}$ copies of the pattern
\begin{equation*}
  P_{c-1}=1,\qquad P_c=\alpha_c,\qquad
  P_i=0\quad(i\notin\{c-1,c\}).
\end{equation*}
Its weight is $(c-1)+c\left(\frac{k}{g_c}-1\right)=\frac{kc}{g_c}-1$, so its defect extension has $\lambda_k=q-\frac{c}{g_c}\ge0$. 
Here nonnegativity follows, with ample room, from $q\ge2k^2$ and
$c/g_c\le k-2$. We refer to the defect type obtained from this construction as a \emph{bridge pattern}.

\medskip
\noindent\textit{Terminal pattern.}
Suppose $R\ne\varnothing$ and put $h:=\max R$.  Acceptance and
\eqref{eq:state-terminal-gcd} give that $\gcd(k,C\cup\{h\})=1$. 
By Lemma~\ref{lem:zero-sum}, choose coefficients $\beta_c\in\{0,\dots,k-1\}$ for $c\in C\cup\{h\}$ such that 
\[
w_*:=\sum_{c\in C\cup\{h\}}c\beta_c\equiv-1\pmod{k}.
\]
Since $k\ge3$ in the terminal case and $C\cup\{h\}\subseteq[k-2]$,
\begin{equation*}
  0<w_*+1
  \le(k-1)\sum_{i=1}^{k-2}i+1
  =\frac{(k-1)^2(k-2)}{2}+1<k^3.
\end{equation*}
Thus $t_*:=(w_*+1)/k$ is a positive integer with $t_*<k^2$.  Put $S_R:=\sum_{i\in R}D_i$ 
and use $S_R$ copies of the pattern whose coordinates are $\beta_c$ on
$C\cup\{h\}$ and zero elsewhere in $M$.  Its defect extension has $\lambda_k=q-t_*\ge0$. We call the resulting defect type a terminal pattern. 
If $R=\varnothing$, use no terminal pattern and set $S_R=0$ and
$\beta_c=0$ for $c\in C$.  When $k=2$, one has $M=\varnothing$, so the terminal
case cannot occur.

\medskip
\noindent\textit{Exact ledger and capacity.}
Let $A_i$ be the aggregate use of level $i\in M$ by all bridge and terminal
defect patterns.  The state partition gives
\begin{align*}
  A_b&=H_b &&(b\in B),\nonumber\\
  A_c&=\alpha_cH_{c-1}+\beta_cS_R &&(c\in C),\nonumber\\
  A_h&=\beta_hS_R &&(R\ne\varnothing,\ h=\max R),\nonumber\\
  A_i&=0 &&(R\ne\varnothing,\ i\in R\setminus\{h\}).
\end{align*}
When $R=\varnothing$, the last two lines are absent.  The number of defect
patterns is exactly
\begin{align}
  \sum_{c\in C}H_{c-1}+S_R
  &=\sum_{c\in C}(D_c+D_{c-1})+\sum_{i\in R}D_i\nonumber\\
  &=\sum_{i\in M}D_i=S.
  \label{eq:defect-count}
\end{align}

We next verify the capacity inequalities
\begin{equation}
  0\le A_i\le H_i\qquad(i\in M).
  \label{eq:capacity-goal}
\end{equation}
Equality holds for $i\in B$, and the assertion is immediate for
$i\in R\setminus\{h\}$.  Fix $c\in C$.  Every $r\in R$ satisfies
$r\le c-2$.  Since $q\ge2k^2$, we have
\[
  n=qk-1\ge2k^3-1\ge2k.
\]
If $c\ge3$, binomial monotonicity gives
$S_R\le |R|\binom{n-1}{c-3}$.  Moreover,
\[
  \frac{H_{c-1}}{H_c}=\frac{c}{n-c+1},
  \qquad
  \frac{\binom{n-1}{c-3}}{H_c}
  =\frac{c(c-1)(c-2)}{n(n-c+1)(n-c+2)}.
\]
Hence
\begin{equation}
  \frac{A_c}{H_c}
  \le
  \alpha_c\frac{c}{n-c+1}
  +\beta_c|R|
   \frac{c(c-1)(c-2)}{n(n-c+1)(n-c+2)}.
  \label{eq:carrier-capacity}
\end{equation}
Here $c\ge3$ forces $k\ge5$.  Since
$n-c+1,n-c+2\ge n/2$ and
$\alpha_c,\beta_c\le k-1$, $|R|\le k-2$, and $c<k$, the two terms on the
right of \eqref{eq:carrier-capacity} are, respectively, less than
\[
  \frac{2k^2}{2k^3-1}<\frac14,
  \qquad
  \frac{4k^5}{n^3}<\frac{4}{k^4}<\frac14.
\]
Thus $A_c<H_c$.  If $c=2$, the order invariant forces $R=\varnothing$, and
\[
  \frac{A_2}{H_2}=\alpha_2\frac{2}{n-1}
  \le\frac{2(k-1)}{n-1}<1.
\]
Finally, if $R\ne\varnothing$, then $D_i\le D_h$ for every $i\in R$, so
\begin{equation*}
  \frac{A_h}{H_h}
  \le\beta_h|R|\frac{D_h}{H_h}
  =\beta_h|R|\frac{h}{n}
  \le\frac{(k-1)(k-2)^2}{n}<1.
\end{equation*}
This also covers $h=1$, and proves \eqref{eq:capacity-goal}.

\medskip
\noindent\textit{Residual blocks.}
Put $r_i:=H_i-A_i\ge0$ for $i\in M$, 
and let $\mathcal P$ be the multiset of the $S$ defect patterns.  By
\eqref{eq:basic-identities} and \eqref{eq:defect-count},
\begin{align*}
  \sum_{i\in M}ir_i
  &=\sum_{i\in M}iH_i-\sum_{P\in\mathcal P}w(P)\nonumber\\
  &=nS-\sum_{P\in\mathcal P}w(P)
    =\sum_{P\in\mathcal P}\bigl(n-w(P)\bigr)
    \equiv0\pmod{k}.
\end{align*}
Thus the multiset containing $r_i$ copies of residue $i$ is a zero-sum
sequence in $\Z/k\Z$.  Partition it using Lemma~\ref{lem:zero-sum}.  If a nonempty
block $Q$ has $\mu_i(Q)$ entries equal to $i$ and total integer weight $kt$,
then its length is at most $k$ and every entry is at most $k-2$. Therefore we have $1\le t\le k-2$. 
Associate to $Q$ the type
\begin{equation*}
  \lambda_i=\mu_i(Q)\ (i\in M),\qquad
  \lambda_{k-1}=1,
  \qquad
  \lambda_k=q-1-t.
\end{equation*}
Its last coordinate is nonnegative, and its total weight is $kt+(k-1)+k(q-1-t)=qk-1=n$. 

Let $N$ be the number of zero-sum blocks, with $N=0$ if the residual multiset
is empty.  Since every block is nonempty,
\[
  N\le\sum_{i\in M}r_i\le\sum_{i\in M}H_i.
\]
If $M\ne\varnothing$, then
\begin{equation*}
  \frac{\sum_{i\in M}H_i}{U}
  \le |M|\frac{H_{k-2}}{H_{k-1}}
  =|M|\frac{k-1}{n-k+2}<1,
\end{equation*}
because $n-k+2\ge2k^3-k+1>(k-2)(k-1)\ge|M|(k-1)$.
Thus $N\le U$; this is automatic when $M=\varnothing$.  Add $U-N$ copies of
the \emph{baseline type} 
\begin{equation*}
  \lambda_i=0\ (i\in M),\qquad
  \lambda_{k-1}=1,
  \qquad
  \lambda_k=q-1.
\end{equation*}

\medskip
\noindent\textit{Closure of all coordinates.}
The defect types contribute $A_i$ to coordinate $i\in M$, and the block types
contribute $r_i$, so the aggregate is $H_i$.  The $N$ block types and $U-N$
baseline types each have $(k-1)$-coordinate one, giving aggregate $U=H_{k-1}$.
The total number of constructed types is $S+N+(U-N)=S+U$. 
Let $K'$ be their aggregate $k$-coordinate.  Since every type has weight $n$,
\begin{align*}
  kK'
  &=n(S+U)-\sum_{i\in M}iH_i-(k-1)U\nonumber\\
  &=n(S+U)-nS-(k-1)U=(n-k+1)U.
\end{align*}
But $H_k/U=(n-k+1)/k$, so $K'=H_k$.  We have constructed a nonnegative
integral solution of every equation in \eqref{eq:type-system}; therefore
Theorem~\ref{thm:type-realization} proves that $\Layers{qk-1}{L}$ is
$1$-factorable.
\end{proof}

\subsection{Rejection implies nonfactorability}

\begin{theorem}\label{thm:proof2}
Under the same hypotheses as in Theorem~\ref{thm:main}, if $M$ is rejected by the gcd-peeling algorithm, then $\binom{[qk-1]}{L}$ is not $1$-factorable. 
\end{theorem}

\begin{proof}
Assume that the algorithm rejects.  At the rejection state put
\begin{equation*}
  h:=\max R,\qquad T:=R\setminus\{h\},
  \qquad d:=\gcd(k,C\cup\{h\})>1.
\end{equation*}
Since $d\mid h$, one has $h\ge 2$.  Rejection also gives $h-1\notin R$, so
$T\subseteq[h-2]$.  The state invariants give
\begin{equation}
  M=C\mathbin{\dot\cup}B\mathbin{\dot\cup}\{h\}
    \mathbin{\dot\cup}T,
  \qquad B=\{c-1:c\in C\}.
  \label{eq:rejection-partition}
\end{equation}
Put $E:=B\cup T$, and for an integral type $\lambda$ define
\[
  g(\lambda):=\lambda_{k-1}+\sum_{i\in E}\lambda_i.
\]
We claim that
\begin{equation}
  g(\lambda)\ge1\qquad(\lambda\in\T(n,L)).
  \label{eq:type-hitting}
\end{equation}
Indeed, if $g(\lambda)=0$, nonnegativity and integrality force every coordinate
in $E\cup\{k-1\}$ to vanish.  The remaining edge sizes belong to
$C\cup\{h,k\}$ and are all divisible by $d$.  The type weight $n$ would then
be divisible by $d$, contrary to $n=qk-1\equiv-1\pmod d$.

Suppose, for a contradiction, that nonnegative real multiplicities $(x_{\lambda})_{\lambda \in \T(n,L)}$ solve
\eqref{eq:type-system}.  By \eqref{eq:total-mass} and
\eqref{eq:type-hitting},
\begin{align*}
  U+\sum_{i\in E}H_i
  &=\sum_\lambda x_\lambda
    \left(\lambda_{k-1}+\sum_{i\in E}\lambda_i\right)\nonumber\\
  &=\sum_\lambda x_\lambda g(\lambda)
    \ge\sum_\lambda x_\lambda=S+U.
\end{align*}
Hence every such solution would imply $\sum_{i\in E}H_i\ge S$. 

We now prove the strict reverse inequality.  Using
\eqref{eq:rejection-partition}, the identity
$H_{c-1}=D_{c-1}+D_c$ for each bridge pair, and
$D_i-H_i=-D_{i+1}$, we obtain
\begin{align}
  S-\sum_{i\in E}H_i
  &=\sum_{c\in C}(D_c+D_{c-1}-H_{c-1})
    +D_h+\sum_{i\in T}(D_i-H_i)\nonumber\\
  &=D_h-\sum_{i\in T}D_{i+1}.
  \label{eq:rejection-gap}
\end{align}
If $T=\varnothing$, this is $D_h>0$.  Otherwise $|T|\le h-2$, and binomial
monotonicity gives $D_{i+1}\le\binom{n-1}{h-2}$ for $i\in T$. 
Furthermore, $\frac{D_h}{\binom{n-1}{h-2}}=\frac{n-h+1}{h-1}$. 
Thus the right-hand side of \eqref{eq:rejection-gap} is positive whenever
\begin{equation}
  n\ge\max\bigl\{2h-2,(|T|+1)(h-1)+1\bigr\}.
  \label{eq:rejection-size-condition}
\end{equation}
Here $2\le h\le k-2$ and $n\ge2k^3-1$, while
\[
  (|T|+1)(h-1)+1\le(h-1)^2+1\le(k-3)^2+1<2k^3-1.
\]
Therefore \eqref{eq:rejection-size-condition} holds. Hence $\sum_{i\in E}H_i<S$, which contradicts $\sum_{i\in E}H_i\ge S$. No nonnegative real
solution of the type equations exists.
\end{proof}

\subsection{Completion of the main theorem and proof of the corollary}

We are now in a position to complete the entire proof.

\begin{proof}[Proof of Theorem~\ref{thm:main}]
    The gcd-peeling algorithm terminates by Lemma~\ref{lem:state-invariants}. The accepting and rejecting cases are dealt with exhaustively by Theorems \ref{thm:proof1} and \ref{thm:proof2}, respectively, which establishes Theorem~\ref{thm:main}.
\end{proof}

\begin{proof}[Proof of Corollary~\ref{cor:hhm-classification}]
Fix $L$ and choose $N_{\mathrm{old}}(L)$ so that all
three conclusions of Proposition~\ref{prop:old-reduction} hold.  Set $N_0(L):=\max\{N_{\mathrm{old}}(L),\,2k^3-1\}$. 
Let $n\ge N_0(L)$.  The residue class $0\pmod{k}$ is settled by
Proposition~\ref{prop:old-reduction}(1), and all classes other than $0$ and $-1$ are
excluded by Proposition~\ref{prop:old-reduction}(2).  In the remaining class write $n=qk-1$ and $q=\frac{n+1}{k}\ge2k^2$. 

If $k-1\notin L$, Proposition~\ref{prop:old-reduction}(3) also excludes the
class $-1\pmod{k}$, so only the class $0\pmod{k}$ remains; this proves
assertion~{\rm(i)}.  Now suppose $k-1\in L$.  If the algorithm accepts, then
Theorem~\ref{thm:main} supplies a factorization throughout the class
$-1\pmod{k}$ above the threshold, and Proposition~\ref{prop:old-reduction}
supplies the class $0\pmod{k}$ while excluding all others.  This is
alternative~{\rm(ii)} and proves assertion~{\rm(ii)}.  If the algorithm
rejects, Theorem~\ref{thm:main} excludes the class $-1\pmod{k}$ as well, so
alternative~{\rm(i)} holds, proving assertion~{\rm(iii)}.  Since $k\ge2$, the
two alternatives are distinct, and the corollary follows.
\end{proof}

\section*{Acknowledgments}

The authors would like to thank Professor Bo Ning for a careful reading of the manuscript and for many helpful suggestions.

\section*{Declaration on the use of generative AI}

The authors used generative AI tools to assist in discussing proof strategies, checking proofs, and
improving the exposition. The authors take full responsibility for the mathematical arguments, results,
and conclusions, all of which were carefully reviewed and verified by them.

\hypersetup{urlcolor=blue}
\begingroup
\setlength{\parindent}{0pt}

\textbf{Guorong Gao}\\
School of Mathematics and Statistics\\
Fuzhou University\\
Fuzhou\\
China\\
E-mail: \href{mailto:grgao@fzu.edu.cn}{grgao@fzu.edu.cn}

\bigskip
\textbf{Mingze Li}\\
School of Mathematical Sciences\\
University of Science and Technology of China\\
Hefei, Anhui 230026, China\\
E-mail: \href{mailto:lmz10@mail.ustc.edu.cn}{lmz10@mail.ustc.edu.cn}

\endgroup

\end{document}